\documentclass[a4paper,12pt]{amsart}
\usepackage{graphicx}
\usepackage{amsmath}
\usepackage{latexsym,lscape,verbatim}
\usepackage{amsthm,float}
\usepackage{epic,latexsym,bezier,caption,amsbsy,psfrag,color,enumerate,amsfonts,amsmath,amscd,amssymb, autograph, epic}
\usepackage{epsfig}
\usepackage{rotating}
\usepackage{graphics}
\usepackage[latin1]{inputenc}

\newtheorem{defi}{Definition}[section]
\newtheorem{theorem}[defi]{Theorem}
\newtheorem{lemma}[defi]{Lemma}
\newtheorem{prop}[defi]{Proposition}
\newtheorem{corollary}[defi]{Corollary}

\theoremstyle{definition}
\newtheorem{example}[defi]{Example}

\newtheorem{remark}[defi]{Remark}

\newcommand{\zz}{{\textcircled z}}

\begin{document}


\title{Permutational powers of a graph}

\author{Matteo Cavaleri}
\address{Matteo Cavaleri, Universit\`{a} degli Studi Niccol\`{o} Cusano - Via Don Carlo Gnocchi, 3 00166 Roma, Italia}
\email{matteo.cavaleri@unicusano.it}

\author{Daniele D'Angeli}
\address{Daniele D'Angeli, Institut f\"{u}r Diskrete Mathematik, Technische Universit\"{a}t Graz \ \ Steyrergasse 30, 8010 Graz, Austria}
\email{dangeli@math.tugraz.at}

\author{Alfredo Donno}
\address{Alfredo Donno, Universit\`{a} degli Studi Niccol\`{o} Cusano - Via Don Carlo Gnocchi, 3 00166 Roma, Italia}
\email{alfredo.donno@unicusano.it}

\begin{abstract}
This paper introduces a new graph construction, the permutational power of a graph, whose adjacency matrix is obtained by the composition of a permutation matrix with the adjacency matrix of the graph. It is shown that this construction recovers the classical zig-zag product of graphs when the permutation is an involution, and it is in fact more general. We start by discussing necessary and sufficient conditions on the permutation and on the adjacency matrix of a graph to guarantee their composition to represent an adjacency matrix of a graph, then we focus our attention on the cases in which the permutational power does not reduce to a zig-zag product. We show that the cases of interest are those in which the adjacency matrix is singular. This leads us to frame our problem in the context of equitable partitions, obtained by identifying vertices having the same neighborhood. The families of cyclic and complete bipartite graphs are treated in details.
\end{abstract}

\maketitle

\begin{center}
\end{center}

\section{Introduction}
Graphs are among the most popular and useful tools used in Mathematics to model aspects of real life. Their relatively simple and natural definition makes these objects very versatile and adaptable in many areas of mathematical research.
Usually graphs are studied via their adjacency matrix, and it is an interesting task to investigate the relationship between the  geometrical properties of a graph and the algebraic properties of the corresponding adjacency matrix. Moreover, there exists a correspondence between operations that can be performed on graphs and operations on matrices. See, for example, \cite{wreathnoi,ijgt, hand, imrich, sabidussi}.
In the last years, a very intriguing product of graphs that has been introduced is the so called \emph{zig-zag product}: iterated applications of this product allow to construct infinite families of expanders.\\
\indent The zig-zag product of two graphs was introduced in \cite{annals}
as a construction allowing to produce, starting from a large graph $G$
and a small graph $H$, a new graph $G\zz H$ which inherits
the size from the large one, the degree from the small one, and
the expansion property from both the graphs. In \cite{annals}, it
is explicitly described how iteration of this construction,
together with the standard squaring, provides an infinite family
of constant-degree expander graphs, starting from a particular
graph representing the building block of this construction (see \cite{lubotullio} for further details on expanders). Topological properties of the zig-zag product have been studied in the paper \cite{zigtree}. It is worth mentioning that the zig-zag product has also interesting connections with Geometric group theory, as it is true that the zig-zag product of the Cayley graphs of two groups returns the Cayley graph of the semi-direct product of the groups \cite{groups}. See also the book \cite{scaralibro} for connections of the zig-zag product with other graph compositions and random walks.

\indent In the present paper, we show that one can in any case replace the role of $G$ by an appropriate permutation matrix $P$ of order $2$. In other words, if $G$ has $k$ vertices, one can realize the adjacency matrix of the zig-zag product $G\zz H$ as the product $\tilde{A}_HP\tilde{A}_H$, where $A_H$ is the adjacency matrix of the graph $H$ (so that $\tilde{A}_H=I_k\otimes A_H$ can be regarded as the adjacency matrix of the graph obtained by taking $k\geq 1$ disjoint copies of $H$), and $P$ is a symmetric permutation matrix. Keeping this in mind, one can ask if it is possible to get the adjacency matrix of some graph (i.e., a symmetric $\tilde{A}_HP\tilde{A}_H$) by considering also permutation matrices $P$ that are not symmetric. The \emph{permutational power graph} corresponds exactly to this case. Now it is clear that, whenever we have $\tilde{A}_HP\tilde{A}_H$ symmetric (and so it can interpreted as the adjacency matrix of an undirected graph) we can ask if the same graph can be obtained via the ``classical" zig-zag product (by using a symmetric $P$). This problem is strictly related to the singularity of the matrix $A_H$. In the case in which $A_H$ is invertible, there is no chance to get a symmetric $\tilde{A}_HP\tilde{A}_H$ with a nonsymmetric $P$. For this reason, we focus our attention on the cases where $A_H$ is singular.\\
\indent In fact, the reason why the graph $H$ is singular says a lot about the nature of its permutational power.
The most general case is when we only know that the adjacency matrix of $H$ is singular. We consider this case in Section \ref{algebra}, proving that a permutation $p$ induces a permutational power of $H$ if its \emph{projection on the range} of the adjacency matrix of $H$ is symmetric. Moreover, a permutational power of $H$ with respect to $p$ can be obtained by a classical zig-zag product if there exists an involution $q$ sharing with $p$ the same projection on the range of the adjacency matrix of $H$. \\
One of the reasons why the adjacency matrix of $H$ may be singular is that there are vertices sharing the same neighborhood: this is one of the cases analyzed in Section \ref{eqp} via the notion of neighborhood equitable partition. Actually, in the framework of random matrices and random graphs, it is conjectured that this is the generic reason why $H$ could be singular (for the symmetric case see \cite{matteo}). The graph obtained collapsing the equivalent vertices  together (in a precise way), can be singular or not. In the latter case every permutational power of $H$ can be obtained via the classical zig-zag product (Theorem \ref{mmm}): the easiest examples are the complete bipartite graphs. On the other extreme, the adjacency matrix of $H$ can be singular even if there are no vertices sharing the same neighborhood: the easiest examples are cyclic graphs $C_n$ with $n$ divisible by $4$.\\
\indent In this paper we show, among other results, that:
\begin{itemize}
\item the symmetry of $\tilde{A}_HP\tilde{A}_H$ only depends on the symmetry of the projection of $P$ on the range of the matrix $\tilde{A}_H$ (see Theorem \ref{proietto} and Corollary \ref{proiettosim});
\item there are permutational powers that do not appear as classical zig-zag products (see Corollary \ref{nozig});
\item we can restrict our attention, in a specific sense, to graphs whose adjacency matrix is singular with pairwise distinct rows
  (see Corollary \ref{teo}, Corollary \ref{coro2} and Theorem \ref{mmm}).
\end{itemize}
Moreover, the cases in which $H$ is a cycle or a complete bipartite graph are studied in details (see Sections \ref{cicli} and \ref{ciclo8}, and Section \ref{bipart}).

\section{Preliminaries and motivations}\label{sectionpreliminaries}

In this paper we will denote by $G = (V_G,E_G)$ a finite undirected graph with vertex set $V_G$ and edge set $E_G$. If
$\{u,v\}\in E_G$, we say that the vertices $u$ and $v$ are adjacent
in $G$, and we write $u\sim v$. Observe that loops and
multi-edges are allowed. A path of length $t$ in $G$ is a sequence
$u_0,u_1,\ldots, u_t$ of vertices such that $u_i\sim u_{i+1}$. The
graph is connected if, for every $u,v\in V_G$, there exists a path
$u_0,u_1,\ldots, u_t$ in $G$ such that $u_0=u$ and $u_t = v$.

Suppose $|V_G|=n$. We denote by $A_G = (a_{u,v})_{u,v\in V_G}$ the
\textit{adjacency matrix} of $G$, i.e., the square matrix of size
$n$ whose entry $a_{u,v}$ is equal to the number of edges joining
$u$ and $v$. The \textit{degree} of a
vertex $u\in V_G$ is defined as $\deg (u) = \sum_{v\in V_G}a_{u,v}$. In
particular, we say that $G$ is {\it regular} of degree $d$, or
$d$-regular, if $\deg(u)=d$, for each $u\in V_G$. For
such a graph $G$, the {\it normalized adjacency matrix} is defined
as $A_G' = \frac{1}{d}A_G$.

We recall now the definition of the zig-zag product of two graphs. This is a noncommutative graph
product producing a graph whose degree depends only on the degree
of the second component graph, and providing large and sparse
graphs. We need some notation.\\
\indent Let $G = (V_G,E_G)$ be a $d$-regular graph
over $n$ vertices. Suppose
that we have a set of $d$ colors (labels), that we identify with
the set $[d]=\{1,2,\ldots,d\}$.
We can assume that, for each vertex $v\in V_G$, the edges incident
to $v$ are labelled by a color $h\in [d]$ near $v$, and that any
two distinct edges issuing from $v$ have a different color near
$v$. This allows to define the \textit{rotation map}
$\textrm{Rot}_{G}:V_G\times [d]\longrightarrow V_G\times [d]$ such
that, for all $v\in V_G$ and $h\in [d]$,
$$
\textrm{Rot}_{G}(v,h) = (w,k)
$$
if there exists an edge joining $v$ and $w$ in $G$, which is
colored by the color $h$ near $v$ and by the color $k$ near $w$.
Note that it may be $h\neq k$. Moreover, the composition $\textrm{Rot}_{G}\circ
\textrm{Rot}_{G}$ is the identity map.

\begin{defi}\label{defizigzag}
Let $G= (V_G,E_G)$ be a $d_G$-regular graph, with $|V_G| =
n$, and let $H = (V_H,E_H)$ be a $d_H$-regular graph such that
$|V_H| = d_G$.
Let $\textrm{Rot}_{G}$ (resp. $\textrm{Rot}_{H}$) denote the
rotation map of $G$ (resp. $H$). The {\it zig-zag product}
$G\zz H$ is the regular graph of degree $d_H^2$ with vertex
set $V_G\times V_H$, identified with the set $V_G\times
[d_G]$, and whose edges are described by the rotation map
$$
\textrm{Rot}_{G\zz H}((v,k),(i,j)) = ((w,l),(j',i')),
$$
for all $v\in V_G, k\in [d_G], i,j \in [d_H]$, if:
\begin{enumerate}
\item $\textrm{Rot}_{H}(k,i) = (k',i')$,
\item $\textrm{Rot}_{G}(v,k') = (w,l')$,
\item $\textrm{Rot}_{H}(l',j) = (l,j')$,
\end{enumerate}
where $w\in V_G$, $l,k',l'\in [d_G]$ and $i',j'\in [d_H]$.
\end{defi}
Observe that labels in $G\zz H$ are elements from $[d_H]^2$.
The vertex set of $G \zz H$ is partitioned into $n$ clouds, indexed by the
vertices of $G$, where by definition the $v$-cloud, with $v\in
V_G$, is constituted by the vertices $(v,1), (v,2), \ldots,
(v,d_G)$. Two vertices $(v,k)$ and $(w,l)$ of $G \zz H$ are
adjacent in $G \zz H$ if it is possible to go from $(v,k)$ to
$(w,l)$ by a sequence of three steps of the following form:
\begin{enumerate}
\item a first step \lq\lq zig\rq\rq within the initial cloud, from the vertex $(v,k)$ to the vertex $(v,k')$,
described by $\textrm{Rot}_{H}(k,i) = (k',i')$;
\item a second step jumping from the $v$-cloud to the $w$-cloud, from the vertex $(v,k')$ to the vertex $(w,l')$, described by
$\textrm{Rot}_{G}(v,k') = (w,l')$;
\item a third step \lq\lq zag\rq\rq within the new cloud, from the vertex $(w,l')$ to the vertex $(w,l)$, described by $\textrm{Rot}_{H}(l',j) =
(l,j')$.
\end{enumerate}

\begin{example} \label{firstexample}
Consider the graphs $G$ and $H$ in Fig. \ref{factorgraphs}.

\begin{figure}[h]
\begin{center}
\psfrag{1}{$1$}\psfrag{2}{$2$}\psfrag{3}{$3$}\psfrag{4}{$4$}

\psfrag{a}{$a$}\psfrag{b}{$b$}\psfrag{c}{$c$}

\psfrag{A}{$A$} \psfrag{B}{$B$}\psfrag{G}{$G$}\psfrag{H}{$H$}
\includegraphics[width=0.6\textwidth]{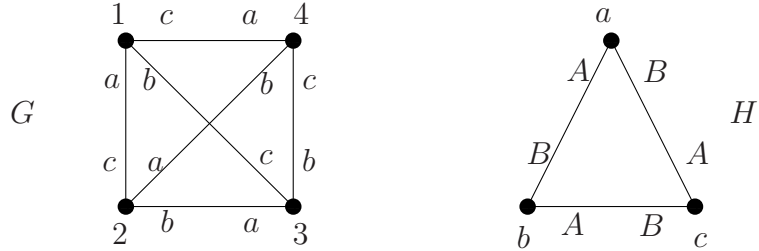} \caption{The graphs $G$ and $H$ of Example \ref{firstexample}.} \label{factorgraphs}
\end{center}
\end{figure}

It follows from Definition \ref{defizigzag} that the graph $G\zz
H$, which is depicted in Fig. \ref{zzproduct}, is the graph with vertex set
$\{1,2,3,4\}\times\{a,b,c\}$, whose edges are labelled by ordered
pairs $(i,j)$ from $\{A,B\}^2$. If we ask, for instance, which are the vertices adjacent to $(1,a)$ in $G\zz H$, we have to take into account that:
$$
\textrm{Rot}_{H}(a,A)= (b,B) \qquad \textrm{Rot}_{H}(a,B) = (c,A)
$$
$$
\textrm{Rot}_{G}(1,b) = (3,c) \qquad \textrm{Rot}_{G}(1,c) = (4,a)
$$
$$
\textrm{Rot}_{H}(c,A)=(a,B) \quad \textrm{Rot}_{H}(c,B)=(b,A) \quad \textrm{Rot}_{H}(a,A)= (b,B) \quad \textrm{Rot}_{H}(a,B) = (c,A).
$$
In this way, we conclude that the vertex $(1,a)$ in $G\zz H$ is adjacent to the vertices $(3,a), (3,b), (4,b), (4,c)$.

\begin{figure}[h]
\begin{center}
\psfrag{1a}{\scriptsize{$(1,a)$}}\psfrag{1b}{\scriptsize{$(1,b)$}}\psfrag{1c}{\scriptsize{$(1,c)$}}
\psfrag{2a}{\scriptsize{$(2,a)$}}\psfrag{2b}{\scriptsize{$(2,b)$}}\psfrag{2c}{\scriptsize{$(2,c)$}}

\psfrag{3a}{\scriptsize{$(3,a)$}}\psfrag{3b}{\scriptsize{$(3,b)$}}\psfrag{3c}{\scriptsize{$(3,c)$}}
\psfrag{4a}{\scriptsize{$(4,a)$}}\psfrag{4b}{\scriptsize{$(4,b)$}}\psfrag{4c}{\scriptsize{$(4,c)$}}
\includegraphics[width=0.37\textwidth]{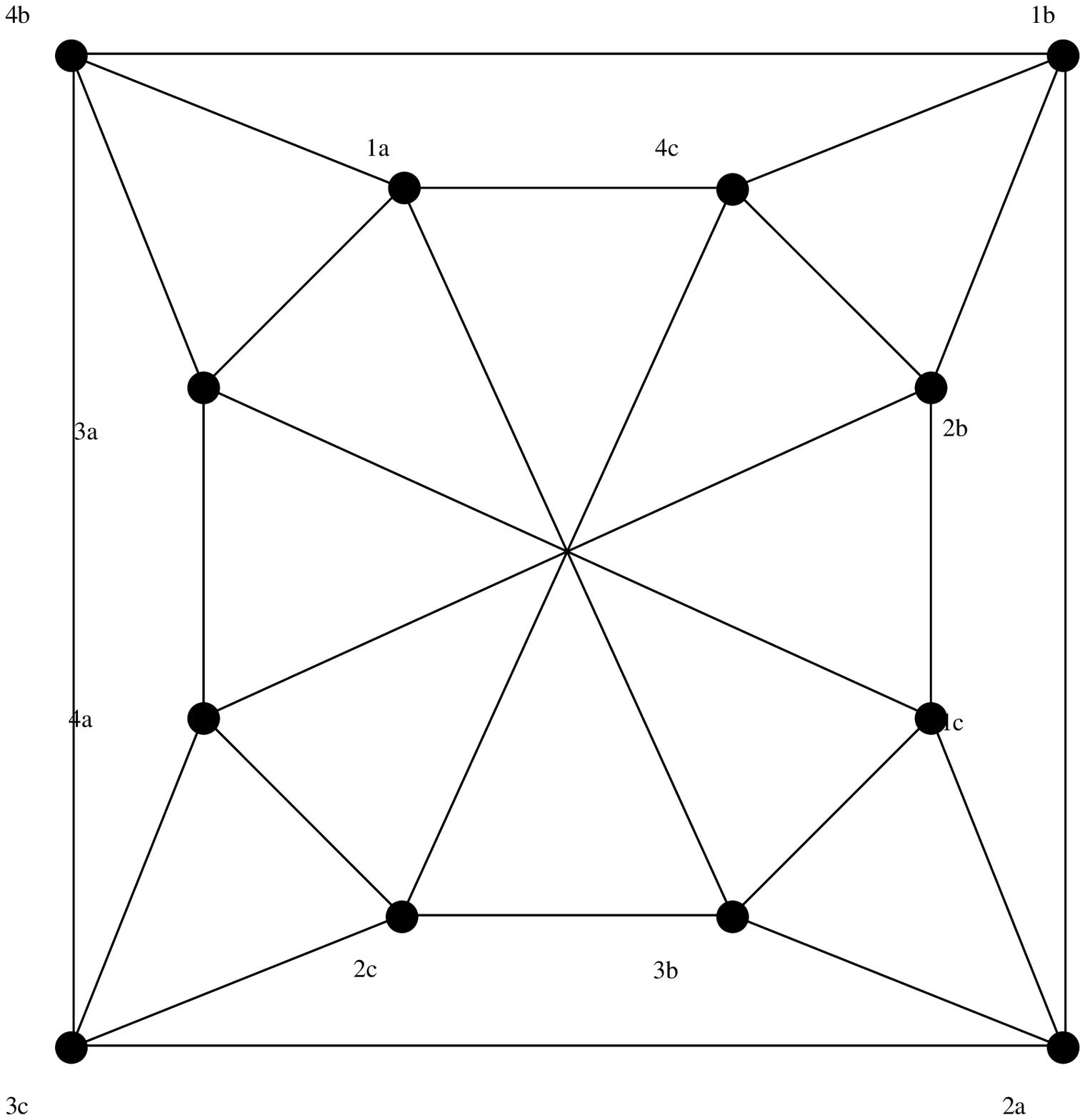} \caption{The graph $G\zz H$ of Example \ref{firstexample}.} \label{zzproduct}
\end{center}
\end{figure}
\end{example}

Let us analyze the adjacency matrix of $G\zz H$. Let $A_{G}$ (resp. $A_{H}$) be the
adjacency matrix of the graph $G$ (resp. $H$). Let $I_n$ denote the identity matrix of size $n$, for each $n\geq 1$. It
follows from the definition of zig-zag product that the adjacency matrix of $G \zz H$ is
$A_{G\zz H}=(I_{|V_G|}\otimes A_{H})P_G(I_{|V_G|}\otimes A_{H})$ (see
\cite{annals}), where $P_{G}$ is the permutation matrix of size $|V_G||V_H|$
associated with the map $\textrm{Rot}_{G}$, i.e.,
$$
{P_G}_{(v,k),(w,l)} =
\begin{cases}
1&\text{if}\ v\sim w \text{ in }G \ \text{by an edge labelled } k\ \text{near}\ v \text{ and } l\ \text{near } w\\
0&\text{otherwise}.
\end{cases}
$$
\noindent If $G$ is the graph of Example \ref{firstexample}, we have
$$
P_G = \left(
                    \begin{array}{ccc|ccc|ccc|ccc}
                       0 & 0 & 0 & 0 & 0 & 1 & 0 & 0 & 0 & 0 & 0 & 0 \\
                       0 & 0 & 0 & 0 & 0 & 0 & 0 & 0 & 1 & 0 & 0 & 0 \\
                       0 & 0 & 0 & 0 & 0 & 0 & 0 & 0 & 0 & 1 & 0 & 0 \\   \hline
                       0 & 0 & 0 & 0 & 0 & 0 & 0 & 0 & 0 & 0 & 1 & 0 \\
                       0 & 0 & 0 & 0 & 0 & 0 & 1 & 0 & 0 & 0 & 0 & 0 \\
                       1 & 0 & 0 & 0 & 0 & 0 & 0 & 0 & 0 & 0 & 0 & 0 \\  \hline
                       0 & 0 & 0 & 0 & 1 & 0 & 0 & 0 & 0 & 0 & 0 & 0 \\
                       0 & 0 & 0 & 0 & 0 & 0 & 0 & 0 & 0 & 0 & 0 & 1 \\
                       0 & 1 & 0 & 0 & 0 & 0 & 0 & 0 & 0 & 0 & 0 & 0 \\    \hline
                       0 & 0 & 1 & 0 & 0 & 0 & 0 & 0 & 0 & 0 & 0 & 0 \\
                       0 & 0 & 0 & 1 & 0 & 0 & 0 & 0 & 0 & 0 & 0 & 0 \\
                       0 & 0 & 0 & 0 & 0 & 0 & 0 & 1 & 0 & 0 & 0 & 0 \\ \end{array}
                  \right).$$
Observe that $P_G^2 = I_{12}$, as the rotation map is an involution. Equivalently, $P_G$ is a symmetric matrix.

For each positive integer $k$, put $\tilde{A}_H = I_k \otimes A_H$ (observe that, for $k=1$, one has $\tilde{A}_H=A_H$).
Our paper moves from the following remark: if we are given a graph $H$, and we consider $k$ copies of such graph, and we are also given a symmetric permutation matrix $P$ of size $k|V_H|$ (i.e., a matrix corresponding to a permutation of order $2$ of $k|V_H|$ elements), then the symmetric matrix $M=\tilde{A}_H P \tilde{A}_H$ can be regarded as the adjacency matrix of a graph composition of type \lq\lq zig-zag\rq\rq, where the jump steps are codified by the matrix permutation $P$. In other words, the first factor graph is not essential to perform the zig-zag construction, one just needs the permutation matrix $P$ describing the rotation map.

\section{Permutational powers of a graph}
The concluding remark of Section \ref{sectionpreliminaries} can be formalized as follows. For every $n\geq 1$, let $Sym(n)$ denote the symmetric group on $n$ elements. Take a regular graph $H$ on $m$ vertices, and fix a positive integer $k\geq 1$. Let $P$ be a symmetric matrix permutation on $km$ elements, that is, the permutation $p\in Sym(km)$ associated with $P$ is the composition of disjoint transpositions (with possibly some fixed elements). Let us identify $V_H$ with the set $\{0,1,\ldots, m-1\}$, and similarly identify the $km$ vertices obtained by taking $k$ copies of $H$ with the set $\{0,1,\ldots, km-1\}$. Observe that each number $x \in \{0,1,\ldots, km-1\}$ admits a unique representation as
$$
x = im + j, \qquad \textrm{with } i=0,1,\ldots,k-1 \ \textrm{ and } j=0,1,\ldots, m-1.
$$
With this interpretation, we can think that the vertex $x$ is the $j$-th vertex belonging to the $i$-th copy of the graph $H$.\\
\indent Let us construct now a labelled $m$-regular graph $G$ on $k$ vertices as follows. Vertices will be named $0,1,\ldots, k-1$, whereas edges will have labels (colors) $0,1,\ldots, m-1$ around each vertex.
More precisely, if the transposition $\tau = (s \ t)$ appears in $p$, with $s = i_sm+j_s$ and $t = i_tm+j_t$, then we connect the vertices $i_s$ and $i_t$ in $G$ by an edge labelled $j_s$ near the vertex $i_s$ and by $j_t$ near the vertex $i_t$. If $u= i_um+j_u$ is an element fixed by $p$, there will be a loop at the vertex $i_u$ with two labels equal to $j_u$. In this situation, the graph with adjacency matrix $\tilde{A}_HP\tilde{A}_H$ coincides with the graph $G\zz H$, where $G$ has been constructed as described above.

\begin{example}\label{senzaprimo}
Consider $3$ copies of the cyclic graph $C_4$ on $4$ vertices, with vertex set $\{0,1,2,3\}$, and the following permutation of order $2$:
$$
p = (0 \ 11) (1 \ 9) (2 \ 5) (3 \ 6) (4\ 10) (7\ 8)
$$
on $12$ elements, and let $P$ be the associated permutation matrix. Then the matrix $(I_3 \otimes A_{C_4})P(I_3\otimes A_{C_4})$ is symmetric, and it is nothing but the adjacency matrix of the zig-zag product $G\zz C_4$, where $G$ is the labelled graph depicted in Fig. \ref{figsenzaprimo}.

\begin{figure}[h]
\begin{center}
\psfrag{a}{$0$}\psfrag{b}{$1$}\psfrag{c}{$2$}
\psfrag{0}{$0$}\psfrag{1}{$1$}
\psfrag{2}{$2$} \psfrag{3}{$3$}
\psfrag{G}{$G$}
\psfrag{C4}{$C_4$} \psfrag{GC4}{$G\zz C_4$}

\psfrag{a0}{$0$} \psfrag{a1}{$4$}\psfrag{a2}{$2$}\psfrag{a3}{$6$}
\psfrag{b0}{$3$} \psfrag{b1}{$11$}\psfrag{b2}{$1$}\psfrag{b3}{$9$}
\psfrag{c0}{$8$} \psfrag{c1}{$7$}\psfrag{c2}{$10$}\psfrag{c3}{$5$}
\includegraphics[width=0.68\textwidth]{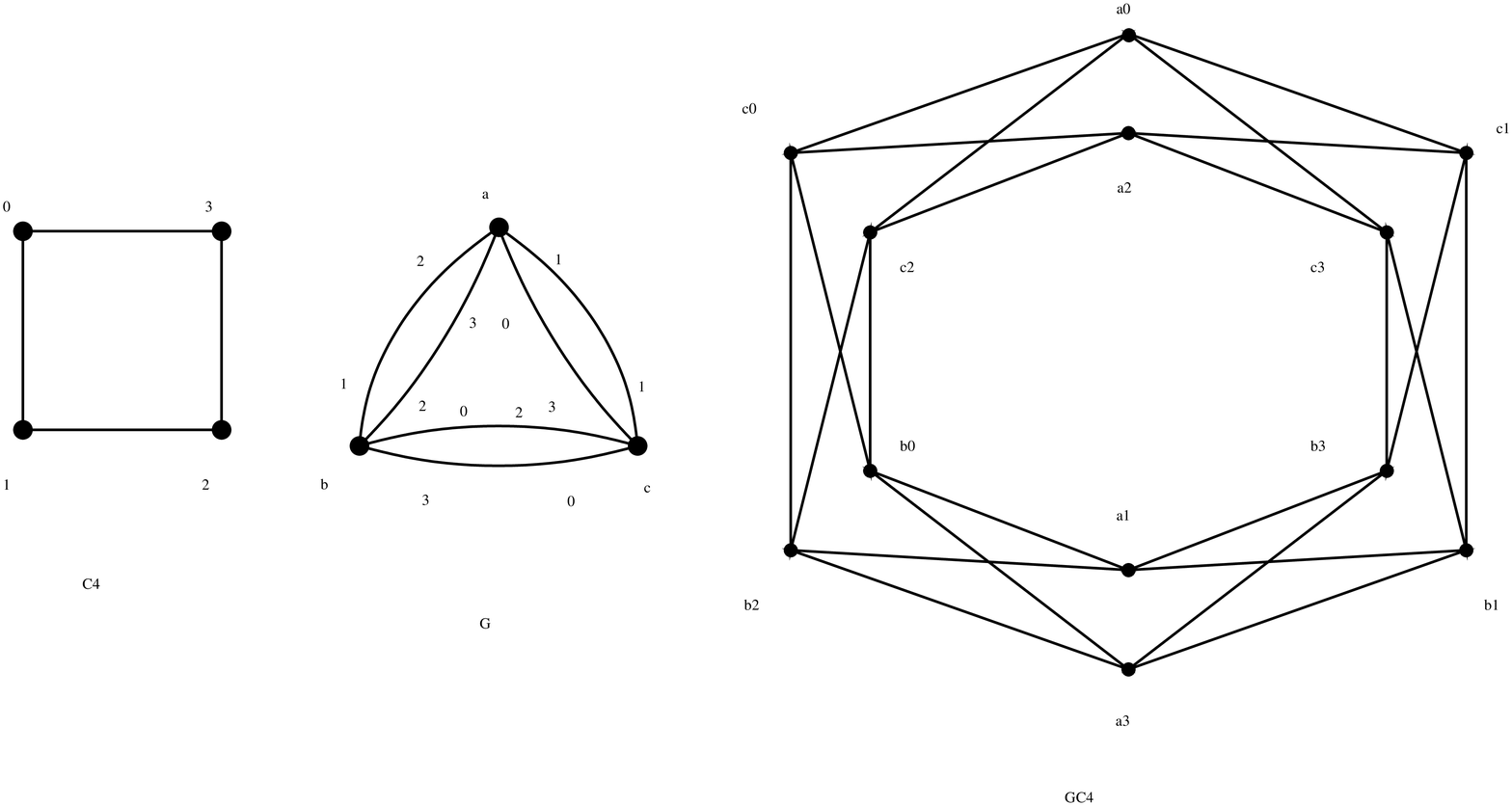} \caption{The graph $G \zz C_4$ of Example \ref{senzaprimo}.} \label{figsenzaprimo}
\end{center}
\end{figure}
\end{example}

\begin{remark}
Even if the zig-zag product has been defined in \cite{annals} for regular graphs, in order to construct increasing sequences of regular expander graphs, our construction shows that one can also start from a nonregular graph $H$. In fact, if one takes $k$ copies of $H$ and a permutation matrix on $k|V_H|$ elements such that the product $M=\tilde{A}_HP\tilde{A}_H$ is symmetric, then $M$ can be regarded as the adjacency matrix of a graph obtained from $H$ by a composition of type \lq\lq zig-zag\rq\rq.
\end{remark}

\begin{example}\label{esempiocaramellina}
In Fig. \ref{straney}, two copies of a nonregular graph $H$ on $6$ vertices, denoted $H_0$ and $H_1$, are represented. Take the permutation $p=(0\ 4\ 6\ 3\ 7\ 5\ 1\ 8)(2\ 9)(10\ 11)$ on $12$ elements. Construct the graph $G$ associated with $p$ (see Fig. \ref{fignonundirected}).
\begin{figure}[h]
\begin{center}
\psfrag{4,0}{$4,0$}\psfrag{3,1}{$3,1$}\psfrag{1,2}{$1,2$}
\psfrag{2,0}{$2,0$}\psfrag{1,5}{$1,5$}
\psfrag{2}{$2$} \psfrag{3}{$3$}
\psfrag{H0}{$H_0$} \psfrag{H1}{$H_1$}
\psfrag{0,3}{$0,3$} \psfrag{4}{$4,5$}\psfrag{5}{$5,4$}\psfrag{0,4}{$0,4$}
\psfrag{5,1}{$5,1$}
\includegraphics[width=0.45\textwidth]{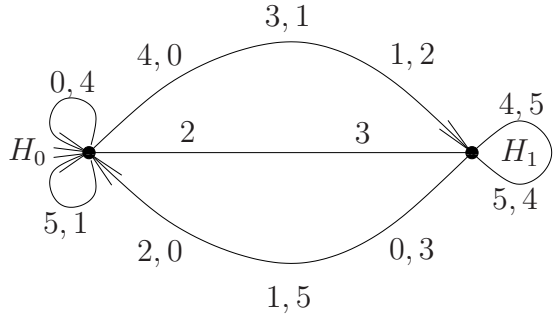} \caption{The graph $G$ of Example \ref{esempiocaramellina}.} \label{fignonundirected}
\end{center}
\end{figure}
The graph resulting from the construction described above appears in the bottom of Fig. \ref{straney}. Notice that it is not a regular graph. If, for instance, we ask which are the vertices in $G\zz H$ which are adjacent to the vertex $2$, we have to think that $2$ is a vertex belonging to the copy indexed by $0$, and its neighbors are the vertices $0,1,3$. Such vertices are mapped by $p$, respectively, to the vertices $4,8,7$. Now, the only neighbor of $4$ in the copy $H_0$ is $3$; the neighbor of $7$ in $H_1$ is $8$; finally, the neighbors of $8$ in $H_1$ are the three vertices $6,7,9$. We then conclude that the vertices adjacent to $2$ are exactly the vertices $3,8,6,7,9$.

\begin{figure}[h]
\begin{center}
\psfrag{4}{$4$}\psfrag{5}{$5$}\psfrag{6}{$6$}
\psfrag{0}{$0$}\psfrag{1}{$1$}
\psfrag{2}{$2$} \psfrag{3}{$3$}

\psfrag{H0}{$H_0$}
\psfrag{H1}{$H_1$} \psfrag{GH}{$G\zz H$}

\psfrag{7}{$7$} \psfrag{8}{$8$}\psfrag{9}{$9$}\psfrag{10}{$10$}
\psfrag{11}{$11$}
\includegraphics[width=0.6\textwidth]{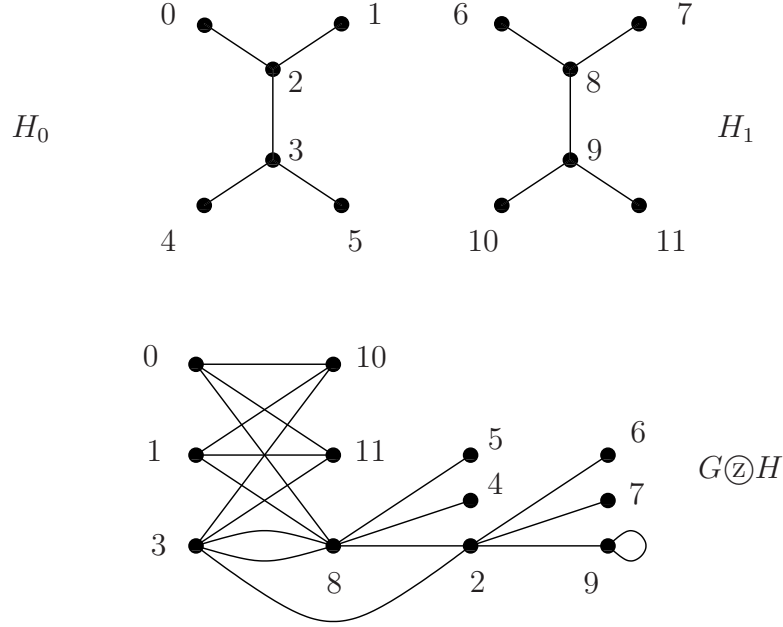} \caption{The graphs $H_0,H_1, G\zz H$ of Example \ref{esempiocaramellina}.} \label{straney}
\end{center}
\end{figure}
\end{example}

Notice that in Example \ref{senzaprimo} the graph $G$ constructed starting from the permutation is undirected, whereas it is directed in Example \ref{esempiocaramellina}, due to the fact that the order of the permutation is not $2$. In Fig. \ref{fignonundirected} the labels in the arc directed from the copy $0$ to the copy $1$ must be interpreted as follows: $4,0$ corresponds to the fact that $p(4)=6\cdot 1 + 0 = 6$; $3,1$ to the fact that $p(3) = 6\cdot 1 + 1=7$ and $1,2$ to the fact that $p(1) =6\cdot 1 + 2= 8$. Similarly, the loop at the $0$ copy starting with $5$ and ending with $1$ corresponds to the fact $p(5)=6\cdot 0 + 1=1$. The undirected edge and the undirected loops correspond to the transpositions $(2\ 9)$ and $(10\ 11)$, respectively.

One can also ask what happens when the matrix $\tilde{A}_H P \tilde{A}_H$ is not symmetric: in this situation, the resulting matrix can be regarded as the signed adjacency matrix of a directed graph, what leads to the possibility of defining a zig-zag product of directed graphs, containing the classical product as a particular case. This general situation will not be investigated in the present paper.\\
\indent On the other hand, it may happen that, even if the permutation $p$ is not of order $2$, anyway one has directed edges from each vertex in the neighborhood of $v$ to each vertex in the neighborhood of $w$ and viceversa, producing an undirected graph. This is the case we are interested in.

\begin{defi}
Let $H=(V_H,E_H)$ be a graph with adjacency matrix $A_H$, and let $p\in Sym (k|V_H|)$, with $k\geq 1$, with associated permutation matrix $P$. Let $\tilde{A}_H = I_k\otimes A_H$. If $M = \tilde{A}_H P \tilde{A}_H$ is symmetric, the graph whose adjacency matrix is $M$ is the \emph{permutational $k$-th power of $H$} with respect to $p$.
\end{defi}

The main questions that we address in our paper are the following.
\begin{enumerate}
\item Given a graph $H$, with adjacency matrix $A_H$, and taken a positive integer $k$, is it possible to find a nonsymmetric permutation matrix $P$ on $k|V_H|$ elements such that the matrix $\tilde{A}_H P \tilde{A}_H$ is symmetric?
\item If this is the case, under which conditions there exists a symmetric permutation matrix $Q$ such that $\tilde{A}_H P \tilde{A}_H = \tilde{A}_H Q \tilde{A}_H$? In other words, when such a permutational power of $H$ can be obtained by the classical zig-zag product?
\end{enumerate}
As an example, observe that the resulting graph in Fig. \ref{figsenzaprimo} can also be obtained by choosing the permutation $p' = (0 \ 11 \ 3 \ 6 \ 10 \ 7\ 8\ 4 \ 1\ 9\ 2\ 5)$ on $12$ elements.
However, we will see that there exist permutational powers of graphs which cannot be obtained by the classical zig-zag construction of Definition \ref{defizigzag} (see Corollary \ref{nozig}).

\subsection{An algebraic interpretation}\label{algebra}
Let us start an algebraic investigation of the symmetry condition
\begin{equation}\label{conditionsimmetry}
(\tilde{A}_HP\tilde{A}_H)^T = \tilde{A}_HP\tilde{A}_H.
\end{equation}

\begin{lemma}\label{lemmainvertible}
If the adjacency matrix $A_H$ of the graph $H$ is invertible, then any permutational power of $H$ is a zig-zag product.
\end{lemma}
\begin{proof}
Observe that the matrix $\tilde{A}_H$ is invertible if and only if the matrix $A_H$ is invertible. Moreover, the matrix $\tilde{A}_H$ is symmetric, as $A_H$ is symmetric. This implies that
$$
(\tilde{A}_HP\tilde{A}_H)^T = \tilde{A}_HP\tilde{A}_H \quad \Longleftrightarrow \quad P^T=P
$$
and so the graph with adjacency matrix $\tilde{A}_HP\tilde{A}_H$ is a zig-zag construction.
\end{proof}

By virtue of Lemma \ref{lemmainvertible}, the interesting cases that it is worth investigating are given by graphs $H$ whose adjacency matrix  is singular. Examples of such graphs are \cite{browerspectrum}:
\begin{itemize}
\item the cyclic graph $C_{n}$ on $n=4h$ vertices;
\item the path graph $P_n$ on $n=2h+1$ vertices;
\item all bipartite graphs with an odd number of vertices;
\item graphs with two or more vertices sharing the same neighborhood (e.g., the complete bipartite graph $K_{m,n}$).
\end{itemize}

Put $M = \tilde{A}_HP\tilde{A}_H$, with $|V_H|=m$. Let $x,y \in\{0,1,\ldots, km-1\}$, with representation
$x = i_xm+j_x$; $y = i_ym+j_y$, with $i_x,i_y \in \{0,1,\ldots, k-1\}$ and $ j_x,j_y\in \{0,1,\ldots, m-1\}$. Then we have:
\begin{eqnarray*}
M_{x,y} &=& \sum_{h,l=0}^{km-1}\widetilde{a}_{x,l}p_{l,h}\widetilde{a}_{h,y}\\
&=& \# \{j_l \sim j_x\ :\ p(i_xm+j_l) = i_ym+j_{l'}, \ j_{l'}\sim j_y\}\\
&=& \#\{j_l \sim j_x\ : \exists j_{l'}\sim j_y \ :\ \ p(i_xm+j_l) = i_ym+j_{l'}\}.
\end{eqnarray*}
If we repeat the same computation for the entry $M_{y,x}$, we deduce that Eq. \eqref{conditionsimmetry} is satisfied if and only if, for each $i_x,i_y\in \{0,1,\ldots, k-1\}$ and
$j_x,j_y\in \{0,1,\ldots, m-1\}$, one has:
$$
\#\{j_l \sim j_x : \exists j_{l'}\sim j_y  : p(i_xm+j_l) = i_ym+j_{l'}\} = \#\{j_{l'} \sim j_y: \exists j_{l}\sim j_x  :  p(i_ym+j_{l'}) = i_xm+j_{l}\}.
$$
In other words, the number of the neighbors $j_l$ of the vertex $j_x$ in the copy $i_x$ of $H$ such that $p$ maps $i_xm+j_l$ to $i_ym+j_{l'}$, where $j_{l'}$ is a neighbor of the vertex $j_y$ in the copy $i_y$, must be equal to the number of the neighbors $j_{l'}$ of the vertex $j_y$ in the copy $i_y$ of $H$ such that $p$ maps $i_ym+j_{l'}$ to $i_xm+j_l$, where $j_l$ is a neighbor of the vertex $j_x$ in the copy $i_x$.

By a similar argument, one can prove the following proposition.
\begin{prop}
Let $H$ be a graph. Let $k$ be a positive integer, and suppose that $P$ is a permutation matrix of size $k|V_H|$ such that the matrix $\tilde{A}_HP\tilde{A}_H$ is symmetric. Then also the matrix $ \tilde{A}_HP^{-1}\tilde{A}_H$ is symmetric.
\end{prop}

The same argument does not apply to the whole cyclic group generated by the permutation $P$, as there exist explicit examples showing that, if the matrix $\tilde{A}_HP\tilde{A}_H$ is symmetric, then the matrix $\tilde{A}_HP^h\tilde{A}_H$ needs not to be symmetric for any integer $h$.

The matrix $A_H$ is symmetric, so that it admits all real eigenvalues, and an orthonormal basis of eigenvectors. Moreover, since we are dealing with a singular matrix, we can assume that $0$ is an eigenvalue for $A_H$. Put $\lambda_0=0$ and let $\lambda_1, \ldots, \lambda_s$ be the nonzero eigenvalues of $A_H$. Let $m_i$ be the multiplicity of $\lambda_i$, for each $i=0,1,\ldots, s$, and let us denote by $E_i$ the corresponding eigenspace, so that $E_0 = \ker A_H$.

Now, the spectrum of the matrix $\tilde{A}_H$ coincides with the spectrum of $A_H$, but the eigenvalue $\lambda_i$ has  multiplicity $km_i$, for each $i=0,1,\ldots,s$. The corresponding eigenspace is $\widetilde{E}_i=\mathbb{R}^k\otimes E_i$. Notice that, due to the symmetry of $A_H$, we have
$$
\mathbb{R}^m = \bigoplus_{i=0}^s E_i = \ker A_H \oplus \left(\bigoplus_{i=1}^s E_i\right),
$$
where the eigenspaces $E_i$ are pairwise orthogonal. Since
$$
\tilde{A}_HP\tilde{A}_H = \tilde{A}_HP^T\tilde{A}_H \quad \Longleftrightarrow \quad \tilde{A}_H(P-P^T)\tilde{A}_H = O,
$$
Eq. \eqref{conditionsimmetry} is satisfied if and only if  the matrix $P-P^T$ maps the space $\bigoplus_{i=1}^s \widetilde{E}_i$ to the space $\widetilde{E}_0 = \mathbb{R}^k\otimes \ker A_H$.

Similarly, given a permutation matrix $P$ such that $\tilde{A}_HP\tilde{A}_H$ is symmetric, and whose order is not $2$, a permutation matrix $Q$ satisfies $\tilde{A}_HP\tilde{A}_H = \tilde{A}_HQ\tilde{A}_H$ if and only if the matrix $P-Q$ maps the space $\bigoplus_{i=1}^s \widetilde{E}_i$ to the space $\widetilde{E}_0 = \mathbb{R}^k\otimes \ker A_H$.

Taking an orthonormal basis from each eigenspace $\widetilde{E}_i$, we construct an orthogonal matrix $U$ such that $U^T \tilde{A}_H U$ is diagonal. Put
$$
U = \left(
      \begin{array}{c|c}
        U_1 & U_0 \\
      \end{array}
    \right)
$$
where $U_1$ is the submatrix whose columns form a basis of the space $\bigoplus_{i=1}^s \widetilde{E}_i$, which is orthogonal to $\widetilde{E}_0$, and $U_0$ is the submatrix whose columns form a basis of  $\widetilde{E}_0$.
\begin{theorem}\label{proietto}
For any two permutation matrices $P$ and $Q$ we have
$$
\tilde{A}_HP\tilde{A}_H=\tilde{A}_HQ\tilde{A}_H \iff U^T_1 P U_1=U^T_1 Q U_1.
$$
\end{theorem}
\begin{proof}
We have $\tilde{A}_HP\tilde{A}_H=\tilde{A}_HQ\tilde{A}_H$ if and only if the matrix $P-Q$ maps the space $\bigoplus_{i=1}^s \widetilde{E}_i$ to the space $\widetilde{E}_0 = \mathbb{R}^k\otimes \ker A_H$. This is equivalent to ask that, for every two columns $u,v$ in $U_1$, one has that $(P-Q)u$ is orthogonal to $v$. Therefore, it must be $U^T_1 (P-Q) U_1=O$, that is the claim.
\end{proof}
\begin{corollary}\label{proiettosim}
The matrix $\tilde{A}_HP\tilde{A}_H$ is symmetric if and only if $U^T_1 P U_1$ is symmetric.
\end{corollary}
\begin{proof}
We can apply  Theorem \ref{proietto} to the case $Q=P^T$.
\end{proof}

From an algebraic point of view, Theorem \ref{proietto} and Corollary \ref{proiettosim} give  the complete answers to our two main questions, and we are going to apply them, in the next subsections, to the case of cyclic graphs. In Section \ref{eqp}, we will be interested in finding more geometric conditions, that in particular cases (e.g., complete bipartite graphs) characterize the permutations $p$ inducing a permutational $k$-th power.

\subsection{Cyclic graphs}\label{cicli} In this section we are going to completely characterize the permutational $1$-st powers of the cyclic graph $C_n$. The adjacency matrix of $C_n$ is the circulant matrix
$$
A_{C_n}=\left(
                \begin{array}{cccccc}
                  0 & 1 & 0& \cdots  & 0 & 1\\
                  1 & 0 & 1&  &  & 0\\
                    0 & 1 & \ddots& \ddots  &&\vdots   \\
                  \vdots & & \ddots  & \ddots & \ddots & 0 \\
                    0 & & & \ddots  &\ddots &1 \\
                  1 & 0& \cdots & 0 &1 &0 \\
                \end{array}
              \right).
$$
Let us denote with $R_n$ the set of complex $n$-th roots of $1$. It is a classical fact  \cite{browerspectrum} that, for  every $\lambda\in R_n$, the vector
$v_\lambda=(\lambda, \lambda^2,\ldots, \lambda^{n-1},1)$
is eigenvector for $A_{C_n}$ of eigenvalue $\lambda+\bar{\lambda}$.


Consider a permutation $p\in Sym(n)$ and the associated permutation matrix $P$. We are going to investigate under which conditions the matrix  $A_{C_n} P A_{C_n}$ is symmetric. As we already noticed, if $n$ is not divisible by $4$, the matrix $A_{C_n}$  is invertible and so by Lemma \ref{lemmainvertible} the matrix $P$ must be symmetric.
From now on, assume that $n=4k$. In the set $[n]$  we define an involution $i\mapsto i^*$, where $i^*$ is the element such that $|i-i^*|=2k$. Since $n$ is even, we have $-\lambda\in R_n$ and $\lambda^{i^*}=-\lambda^{i}$.
\begin{lemma}\label{coro}
Let $\zeta$ be an $n$-th primitive root of $1$, and let $i_1,i_2,i_3,i_4\in[n]$. Then:
$$
\zeta^{i_1}+ \zeta^{i_2}= \zeta^{i_3}+ \zeta^{i_4}\implies \left( i_1=i^*_2 \wedge  i_3=i^*_4\right) \vee \{i_1,i_2\}=\{i_3,i_4\}.
$$
\end{lemma}
\begin{proof}
Set $\lambda_1=\zeta^{i_1}$, $\lambda_2=\zeta^{i_2}$, $\lambda_3=\zeta^{i^*_3}$, $\lambda_4=\zeta^{i^*_4}$, then we have $\lambda_1+ \lambda_2+\lambda_3+ \lambda_4=0,$ with $\lambda_1, \lambda_2, \lambda_3, \lambda_4$ roots of unity. As a consequence of Theorem 6 in \cite{mann}, the only possibility is that
$(\lambda_1=-\lambda_2 \wedge \lambda_3=-\lambda_4)\vee (\lambda_1=-\lambda_3 \wedge \lambda_2=-\lambda_4)\vee (\lambda_1=-\lambda_4 \wedge \lambda_2=-\lambda_3)$.
Since  $\zeta$ is primitive, for $i,j\in [n]$, we have $\zeta^i=\zeta^j\implies i=j$: the thesis follows.
\end{proof}

In the spirit of the introductory remarks to Theorem \ref{proietto}, the matrix $A_{C_n} P A_{C_n}$ is symmetric  if and only if $(P-P^T)v_\lambda\in \ker A_{C_n}$, for each $ \lambda\in R_n\setminus \{\pm i\}$, since the vectors $\{v_\lambda: \lambda\in R_n, \lambda\neq \pm i \}$ are a basis for the range of $A_{C_n}$. Now we have:
$$
(P-P^T)v_\lambda=
\left(
                \begin{array}{c}
              \lambda^{p(1)}- \lambda^{p^{-1}(1)}\\
          \lambda^{p(2)}- \lambda^{p^{-1}(2)}\\
               \vdots\\
           \lambda^{p(n-1)}- \lambda^{p^{-1}(n-1)}\\
   \lambda^{p(n)}- \lambda^{p^{-1}(n)}
                \end{array}
              \right)
=\left(
                \begin{array}{c}
              \lambda^{p(1)}+\lambda^{p^{-1}(1)^*}\\
          \lambda^{p(2)}+\lambda^{p^{-1}(2)^*}\\
               \vdots\\
           \lambda^{p(n-1)}+ \lambda^{p^{-1}(n-1)^*}\\
   \lambda^{p(n)}+ \lambda^{p^{-1}(n)^*}
                \end{array}
              \right).
$$
The condition $(P-P^T)v_\lambda \in \ker A_{C_n}$ is equivalent to:
$$
\begin{cases}
\lambda^{p(2)}+\lambda^{p^{-1}(2)^*}+ \lambda^{p(n)}+\lambda^{p^{-1}(n)^*} =0\\
    \lambda^{p(1)}+\lambda^{p^{-1}(1)^*}+ \lambda^{p(3)}+\lambda^{p^{-1}(3)^*}=0\\
     \vdots\\
      \lambda^{p(2i)}+\lambda^{p^{-1}(2i)^*}+ \lambda^{p(2i+2)}+\lambda^{p^{-1}(2i+2)^*}=0\\
 \lambda^{p(2i-1)}+\lambda^{p^{-1}(2i-1)^*}+ \lambda^{p(2i+1)}+\lambda^{p^{-1}(2i+1)^*}=0\\
      \vdots\\
       \lambda^{p(n-2)}+\lambda^{p^{-1}(n-2)^*}+ \lambda^{p(n)}+\lambda^{p^{-1}(n)^*}=0\\
     \lambda^{p(1)}+\lambda^{p^{-1}(1)^*}+ \lambda^{p(n-1)}+\lambda^{p^{-1}(n-1)^*}=0.\\
  \end{cases}
 $$
By suitably coupling these equations, we get
$$
\begin{cases}
    \lambda^{p(i)}+\lambda^{p^{-1}(i)^*}= \lambda^{p(j)}+\lambda^{p^{-1}(j)^*}\quad
  &i \equiv j \mod 4 \\
  \lambda^{p(i)}+\lambda^{p^{-1}(i)^*}= -(\lambda^{p(j)}+\lambda^{p^{-1}(j)^*}) \quad
 &i \equiv j+2 \mod 4. \\
\end{cases}
$$
If $\zeta\in R_n$ is primitive, by virtue of Lemma \ref{coro}, the condition $(P-P^T)v_\zeta \in \ker A_{C_n}$ is equivalent to:\\
\begin{equation}\label{M0}
\begin{split}
\left(p(i),p(j) \right) &= \left( p^{-1}(i),p^{-1}(j) \right) \mbox{ or }\\
\left(p(i),p(j) \right)&=\left(p^{-1}(j)^*\!, p^{-1}(i)^*\right);
\end{split}
\end{equation}
for    $i \equiv j \mod 4$;

\begin{equation}\label{M2}
\begin{split}
\left(p(i),p(j)\right)&=(p^{-1}(i),p^{-1}(j)) \mbox{ or }\\
\left(p(i),p(j)\right)&=(p^{-1}(j),p^{-1}(i)) \mbox{ or }\\
\left(p(i),p^{-1}(i)\right)&=\left(p(j)^*,p^{-1}(j)^*\right)
\end{split}
\end{equation}
for $i \equiv j+2 \mod 4$. This is true because, for $i\neq j$, the condition  $\{p(i),p^{-1}(i)^*\}= \{p(j),p^{-1}(j)^*\}$ of Lemma \ref{coro} implies $ p(i)=p^{-1}(j)^*$ and  $p(j)=p^{-1}(i)^*.$

We are now in position to prove the following theorem.
\begin{theorem}\label{teo8}
If $n>8$,
$$
A_{C_n} P A_{C_n}= (A_{C_n} P A_{C_n})^T \iff P=P^T.
$$
In other words, there is no permutational $1$-st power of $C_n$ with respect to a nonsymmetric permutation.
\end{theorem}
\begin{proof}
By contradiction, suppose  $P\neq P^T$. Without loss of generality $p(1)\neq p^{-1}(1)$. By Eq. \eqref{M0} we have
$p(1)=p^{-1}(5)^*$ and $p(1)=p^{-1}(9)^*$: it implies  $p^{-1}(5)=p^{-1}(9)$, that is impossible.
\end{proof}

Observe that the graph $C_4$ belongs to the class of  complete bipartite graphs,  that we are going to analyze in detail in Section \ref{bipart}. The remaining case is the cycle graph $C_8$, which is studied in the next section.

\subsection{The cycle $C_8$} \label{ciclo8}
In order to obtain necessary conditions for the symmetry of the matrix $A_{C_8}PA_{C_8}$, we study the behavior of the permutation $p^2$. From the first line of Eq. \eqref{M0} and the first line of Eq. \eqref{M2} we have that if  $i \equiv j \mod 2$ then
$$
p^2(i)=i \iff p^2(j)=j.
$$
Then if $A_{C_8}PA_{C_8}$ is symmetric, we only have the following $4$ possibilities:
\begin{enumerate}[a)]
\item $\{i\in [8]: p^2(i)=i\}=[8]$
\item $\{i\in [8]: p^2(i)=i\}=\{2,4,6,8\}$
\item $\{i\in [8]: p^2(i)=i\}=\{1,3,5,7\}$
\item $\{i\in [8]: p^2(i)=i\}=\emptyset$.
\end{enumerate}

The case a) concerns permutations of order $2$ and then the classical zig-zag product.\\
In the case b), the even numbers are in 1-cycles or 2-cycles and there is a 4-cycle containing the odd numbers. We know that $p(1)\neq p^{-1}(1)$. Considering $i=1$ and $j=3$ in Eq. \eqref{M2} we get
$p(1)=p^{-1}(3)$ or $p(1)=p(3)^*$. On the other hand, Eq. \eqref{M0} with $i=3$ and $j=7$ gives $p(3)^*=p^{-1}(7)$, so that it must be $p^2(1)$ equal to $3$ or $7$. By an analogous argument we have:
\begin{equation}\label{valori1}
\{p^2(1),p^2(5)\}=\{3,7\}
\mbox{ and }
\{p^2(3),p^2(7)\}=\{1,5\}.
\end{equation}
As a consequence, the  only 4-cycles that could appear in the case b) are:
$$\sigma_1=(1\,3\,7\,5), \sigma_2=(1\,7\,3\,5), \sigma_3=(1\,5\,3\,7), \sigma_4=(1\,5\,7\,3).$$
By an explicit computation (our conditions, a priori, are only necessary) one can check
that $A_{C_8}PA_{C_8}$ is symmetric when $p=\sigma_i$ (and therefore also when $p$ is a product of $\sigma_i$ with a permutation of order $2$ of the even numbers).

An analogous argument in the case c) gives
\begin{equation}\label{valori2}
\{p^2(2),p^2(6)\}=\{4,8\}
\mbox{ and }
\{p^2(4),p^2(8)\}=\{2,6\}
\end{equation}
and the $4$-cycles $$\tau_1=(2\,4\,8\,6),\tau_2=(2\,8\,4\,6),\tau_3=(2\,6\,4\,8),\tau_4=(2\,6\,8\,4).$$

So we have that $p$ is a permutation of case c) if and only if $p$ is the product of a $\tau_i$ with a permutation of order $2$ of the odd numbers.

In the case d) we can apply Eq. \eqref{M0} to all pairs $(i,i^*)$, obtaining
\begin{equation}\label{coniug}
p(i)=(p^{-1}(i^*))^*,\quad \forall i \in [n];
\end{equation}
that is, $p$ is conjugated to its inverse by the permutation induced by the involution $^*$.
Moreover, we have that  Eq. \eqref{valori1} and Eq. \eqref{valori2} hold. In particular $p^4$ should be the identity or an involution and therefore $p$ is an $8$-cycle or the product of two $4$-cycles. Moreover, if $p$ were a $8$-cycle, we would have that $p^4(i)=i^*$ for any $i\in[n],$ that is $p=(a\, b\, c\, d\, a^*\,  b^*\, c^*\, d^*)$ for some $a,b,c,d \in [n]$. By applying Eq. \eqref{coniug} to $p$ we have $(a\, b\, c\, d\, a^*\,  b^*\, c^*\, d^*)=(d\, c\, b\, a\, d^*\,  c^*\, b^*\, a^*)$ that it is impossible. Thus $p$ is a product of two disjoint $4$-cycles.
If the permutation sends even (resp. odd) numbers in even (resp. odd) numbers, we are just in the case $p=\tau_i \sigma_j$. If this is not the case, the cycles alternate odd with even numbers; by Eq. \eqref{valori1} and Eq. \eqref{valori2} such a permutation must be of the form:
$$
p=(1\, a\, b \, c) ( 5\, d\, e \, f),
$$
with  $a,c,d,f$ even, $\{b,e\}=\{3,7\}$, $a^*\neq c$ and $d^*\neq f$. Finally, applying Eq. \eqref{coniug} for $i=1,5,a,$ we have that $f=a^*$, $d=c^*$ and $e=b^*$, and therefore the permutation $p$ should be of the form $p=(1\, a\, b \, c) ( 5\, c^*\, b^* \, a^*),$
where we can freely choose $b\in \{3,7\}$, $a\in \{2,4,6,8\}$ and $c$ between the two even numbers different from $a^*$. \\The corresponding $16$ permutations are:
\begin{align*}
&\gamma_1=(1\, 2\, 3 \, 4) ( 5\, 8\, 7 \, 6)\quad
\gamma_2=(1\, 4\, 3 \, 6) ( 5\, 2\, 7 \, 8)\quad
\gamma_3=(1\, 6\, 3 \, 8) ( 5\, 4\, 7 \, 2)\quad
\gamma_4=(1\, 8\, 3 \, 2) ( 5\, 6\, 7 \, 4)&\\
&\gamma_5=(1\, 2\, 7 \, 4) ( 5\, 8\, 3 \, 6)\quad
\gamma_6=(1\, 4\, 7 \, 6) ( 5\, 2\, 3 \, 8)\quad
\gamma_7=(1\, 6\, 7 \, 8) ( 5\, 4\, 3 \, 2)\quad
\gamma_8=(1\, 8\, 7 \, 2) ( 5\, 6\, 3 \, 4)&\\
&\gamma_9=(1\, 2\, 3 \, 8) ( 5\, 4\, 7 \, 6)\quad
\gamma_{10}=(1\, 4\, 3 \, 2) ( 5\, 6\, 7 \, 8)\quad
\gamma_{11}=(1\, 6\, 3 \, 4) ( 5\, 4\, 7 \, 2)\quad
\gamma_{12}=(1\, 8\, 3 \, 6) ( 5\, 2\, 7 \, 4)&\\
&\gamma_{13}=(1\, 2\, 7 \, 8) ( 5\, 4\, 3 \, 6)\quad
\gamma_{14}=(1\, 4\, 7 \, 2) ( 5\, 6\, 3 \, 8)\quad
\gamma_{15}=(1\, 6\, 7 \, 4) ( 5\, 8\, 3 \, 2)\quad
\gamma_{16}=(1\, 8\, 7 \, 6) ( 5\, 2\, 3 \, 4).&
\end{align*}
By a direct computation, one can check that all these permutations make $A_{C_8}PA_{C_8}$ symmetric.

By summarizing, the permutations $p\in Sym(8)$ such that $p^2\neq Id$ and $A_{C_8}PA_{C_8}$ is symmetric are the following $112$ permutations: $$\quad q\sigma_i,\quad  s\tau_j,\quad  \sigma_i\tau_j, \quad \gamma_k,$$ where $i,j=1,\ldots, 4$, $k=1,\ldots, 16$, and $q$ (resp. $s$) is a permutation of order at most $2$ fixing each odd (resp. even) number.

\begin{example}\label{ciclo88}
There exist permutational $1$-st powers of $C_8$ that can be obtained by a permutation of order $2$ and there exist permutational $1$-st powers of $C_8$ for which this is impossible.
For instance, consider the permutation $p_1=\tau_4=(2\,6\,8\,4)$ with associated matrix $P_1$, and suppose $q_1\in Sym(8)$ is such that  $A_{C_8}Q_1A_{C_8}=A_{C_8}P_1A_{C_8}$. The analogue of Eq. \eqref{M0} for $p_1$ and $q_1$ gives $(q_1(2),q_1(6))=(6,8)$ or $(q_1(2),q_1(6))=(8^*,6^*)=(4,2)$. In both cases $q_1^2\neq Id$. The graph with adjacency matrix  $A_{C_8}P_1A_{C_8}$  is depicted in Fig. \ref{figurafinalexample} (observe that it consists of two connected components).\\
On the other hand, let $p_2=\gamma_1=(1\, 2\, 3 \, 4) ( 5\, 8\, 7 \, 6)$, with associated matrix $P_2$. Then one can check that, by defining $q_2=(1\, 2)(3 \, 4) ( 5\, 8)( 7 \, 6)$, one has $A_{C_8}P_2 A_{C_8}=A_{C_8}Q_2 A_{C_8}$, showing that both the possibilities may occur.
\begin{figure}[h]
\begin{center}
\psfrag{4}{$4$}\psfrag{5}{$5$}\psfrag{6}{$6$}\psfrag{7}{$7$}
\psfrag{8}{$8$}\psfrag{1}{$1$}
\psfrag{2}{$2$} \psfrag{3}{$3$}  \psfrag{C8}{$C_8$}
\includegraphics[width=0.6\textwidth]{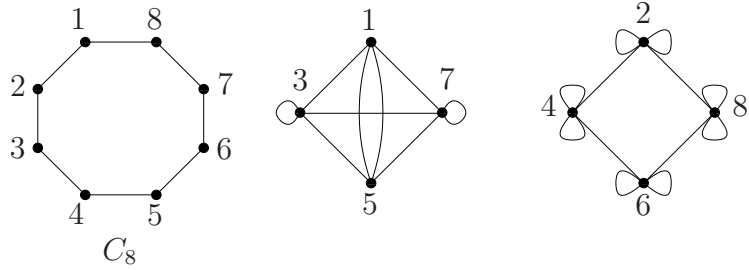} \caption{The graph $C_8$ and its permutation $1$-st power with respect to $p_1$.} \label{figurafinalexample}
\end{center}
\end{figure}
\end{example}
\begin{corollary}\label{nozig}
There exist permutational powers that cannot be expressed as zig-zag product.
\end{corollary}

\section{Equitable partitions}\label{eqp}
In this section we will use the notion of equitable partition. The main idea is that, whenever we declare equivalent vertices of the graph sharing
the same set of neighbors, we get a partition of the vertex set which is equitable. This observation allows us to deeply explore the structure of the permutations giving rise to permutational powers and, in particular, to detect those permutational powers that actually can be obtained by a classical zig-zag product.\\
\indent Equitable partitions have a number of significant applications in Graph theory: for example, the vertex set partition of a graph under the action of a group of automorphisms is always equitable. This fact has been used in the context of graph isomorphism algorithms (we refer to the book \cite{godsil} for more details).

Let $G=(V_G,E_G)$ be a graph, with $|V_G|=n$. With a given partition $\pi=\{C_1,\ldots, C_m\}$ of $V_G$ we can associate an $n\times m$ matrix $M_\pi$, called the \emph{characteristic matrix} of $\pi$, defined as follows. For each $v\in V_G$, and $i\in\{1,\ldots, m\}$, one has
\begin{equation*}
{(M_\pi)}_{v,i}=
\left\{
  \begin{array}{ll}
\frac{1}{\sqrt{|C_i|}} & \hbox{if } v\in C_i\\
  0 & \hbox{otherwise.}
  \end{array}
\right.
\end{equation*}
It is easy to check that:
\begin{equation}\label{ident}
M_\pi^TM_\pi=I_m.
\end{equation}
Moreover we can use the characteristic matrix of $\pi$ to define the $m\times m$ matrix $A_G/\pi=M_\pi^T A_G M_\pi$. This matrix represents the restriction of the matrix $A_G$ to the parts of $\pi$.
\begin{defi}\cite{schwenk}\label{equi}
For each vertex $v\in V_G$, put $B_1(v)= \{u\in V_G: v\sim u\}$, that is, $B_1(v)$ is the neighborhood of $v$ in $G$. Then a partition $\pi$ of $V_G$ is \emph{equitable} if, for all $i\in \{1,\ldots, m\}$, and all $v,w\in C_i$, one has $|B_1(v)\cap C_j|=|B_1(w)\cap C_j|$, for each $j\in \{1,\ldots, m\}$.
\end{defi}
It is known that the equitability condition of Definition \ref{equi} is equivalent to each of the following (the reader can refer to \cite{godsil80, godsil, godsil2}):
\begin{itemize}
\item
$A_GM_\pi=M_\pi (A_G/\pi)$;
\item
$M^T_\pi A_G=(A_G/\pi)M^T_\pi$;
\item
$A_G M_\pi M_\pi^T=M_\pi M_\pi^T A_G$.
\end{itemize}
Moreover, if $\pi$ is equitable, one has
$$
(M_\pi^TA_GM_\pi)_{i,j}=\sum_{h,l=1}^n m_{h,i}a_{h,l}m_{l,j} = \sum_{h\sim l}m_{h,i}m_{l,j},
$$
so that the entry $(A_G/\pi)_{i,j}$ equals the total number of edges connecting a vertex of $C_i$ to a  vertex of $C_j$, multiplied by $\frac{1}{\sqrt{|C_i||C_j|}}$. In other words, the matrix $A_G/\pi$ can be regarded, up to a suitable normalization, as the adjacency matrix of the quotient graph $G/\pi$ obtained from $G$ by taking the quotient of $V_G$ modulo the equivalence relation defined by $\pi$.

In what follows we will mostly focus on a very specific equitable partition $\hat\pi$, the one induced by the relation ``to have the same neighborhood''. This partition is natural in the context of graphs and fits
into our setting of permutational powers of $G$. It concretely corresponds to the existence of two or more rows in the matrix $A_G$ which are equal. This condition assures that $A_G$ is not invertible and this is exactly the case we are interested in, by virtue of Lemma \ref{lemmainvertible}.

\begin{defi}\label{cappuccio}
Let $G=(V_G,E_G)$. The partition $\hat\pi$ is the partition of $V_G$  such that $v, w\in V_G$ are in the same part if $B_1(v)=B_1(w)$.
\end{defi}
By definition, $\hat\pi$ is an equitable partition; we will call it the \emph{neighborhood partition} of $V_G$. The following proposition shows that $\hat\pi$ has an even stronger property.

\begin{prop}
Let $\hat\pi$ be the neighborhood partition of $V_G$. Then
\begin{equation}\label{vicini}
A_G M_{\hat\pi}M_{\hat\pi}^T=A_G =M_{\hat\pi}M_{\hat\pi}^TA_G.
\end{equation}
\end{prop}
\begin{proof}
For all $u,v\in V_G$, we have:
$$
(M_{\hat\pi}M_{\hat\pi}^T)_{u,v}=
\left\{
  \begin{array}{ll}
   \frac{1}{|C_{u,v}|} & \hbox{if $u$ and $v$ belong to the same class $C_{{u,v}}$}; \\
   0 & \hbox{otherwise}.
  \end{array}
\right.
$$
This implies:
\begin{eqnarray*}
(M_{\hat\pi}M_{\hat\pi}^TA_G)_{u,v}&=& \sum_{w\in V_G}(M_{\hat\pi}M_{\hat\pi}^T)_{u,w}a_{w,v}
= \sum_{w\in C_{{u,w}}} a_{w,v}\frac{1}{|C_{{u,w}}|}
= \sum_{w\in C_{{u,w}}} a_{u,v}\frac{1}{|C_{{u,w}}|}\\
&=&a_{u,v}\frac{|C_{{u,w}}|}{|C_{{u,w}}|}
=(A_G)_{u,v},
\end{eqnarray*}
where we have used that $a_{w,v}=a_{u,v}$, because $w$ and $u$ are in the same part of $\hat{\pi}$.
Since $\hat\pi$ is equitable, we conclude that $A_GM_{\hat\pi}M_{\hat\pi}^T=M_{\hat\pi}M_{\hat\pi}^TA_G=A_G$.
\end{proof}

In particular $G/{\hat\pi}$ is the graph where the vertices of $G$ with the same neighborhood are identified.

\begin{example}\label{vicinati}
Let $H$ be the graph on $6$ vertices, with $V_H=\{0,1,2,3,4,5\}$, of Example \ref{esempiocaramellina}. We have  $\hat{\pi}=\{C_1,C_2,C_3,C_4\}$, where $C_1=\{0,1\}$, $C_2=\{2\}$, $C_3=\{3\}$, and $C_4=\{4,5\}$. The adjacency matrix of $H$ and the characteristic matrix of $\hat{\pi}$ are, respectively:
$$
A_H=\left(
    \begin{array}{cccccc}
      0 & 0 & 1 & 0 & 0 & 0 \\
      0 & 0 & 1 & 0 & 0 & 0 \\
      1 & 1 & 0 & 1 & 0 & 0 \\
      0 & 0 & 1 & 0 & 1 & 1 \\
      0 & 0 & 0 & 1 & 0 & 0 \\
      0 & 0 & 0 & 1 & 0 & 0 \\
    \end{array}
  \right) \qquad  M_{\hat{\pi}} = \left(
            \begin{array}{cccc}
              1/\sqrt{2} & 0 & 0 & 0 \\
              1/\sqrt{2} & 0 & 0 & 0 \\
              0 & 1 & 0 & 0 \\
              0 & 0 & 1 & 0 \\
              0 & 0 & 0 & 1/\sqrt{2} \\
              0 & 0 & 0 & 1/\sqrt{2} \\
            \end{array}
          \right).
$$
A direct computation gives
$$
M_{\hat{\pi}}^TA_HM_{\hat{\pi}} = \left(
                                          \begin{array}{cccc}
                                            0 & \sqrt{2} & 0 & 0 \\
                                            \sqrt{2} & 0 & 1 & 0 \\
                                            0 & 1 & 0 & \sqrt{2} \\
                                            0 & 0 & \sqrt{2} & 0 \\
                                          \end{array}
                                        \right),
$$
which is, up to normalization, the adjacency matrix of the quotient graph $H/\hat{\pi}$ in Fig. \ref{figurequotientgraph}.
\begin{figure}[h]
\begin{center}
\psfrag{C1}{$C_1$}\psfrag{C2}{$C_2$}\psfrag{C3}{$C_3$}\psfrag{C4}{$C_4$}
\includegraphics[width=0.5\textwidth]{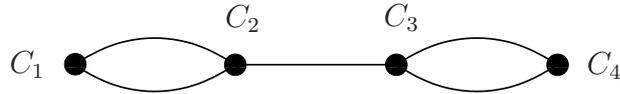} \caption{The quotient graph $H/\hat{\pi}$ of Example \ref{vicinati}.} \label{figurequotientgraph}
\end{center}
\end{figure}
\end{example}

We focus now our attention on the case of a graph $H$ for which we want to investigate permutational $k$-th powers. Observe that the neighborhood partition $\hat\pi$ can be considered also for the graph obtained by taking $k$ disjoint copies of $H$, and it is induced in a very natural way by the neighborhood partition of the vertex set $V_H$. As usual, we put $\tilde{A}_H=I_k\otimes A_H$, and $P$ is a matrix permutation acting on $k|V_H|$ elements, with $k\geq 1$.

\begin{prop}\label{onlysufficient}
Let $H=(V_H,E_H)$ be a graph with adjacency matrix $A_H$. Then vertices having the same neighborhood in $H$ produce ($k$ copies of) vertices having the same neighborhood in any permutational $k$-th power of $H$.
\end{prop}
\begin{proof}
Put $|V_H|=n$. Let $A_H=(a_{i,j})_{i,j=0,\ldots, n-1}$ be the adjacency matrix of $H$ and assume that $a_{r,j}= a_{s,j}$ for each $j=0,\ldots, n-1$, that is, the $r$-th and the $s$-th vertex of $H$ have the same neighborhood in $H$. Consider the permutational $k$-th power of $H$ induced by a permutation $p$, with associated permutation matrix $P$, and enumerate its vertices as $0,1,\ldots, kn-1$. Then, for each $h=0,\ldots, k-1$ and $j= 0,\ldots, kn-1$, we have:
\begin{eqnarray*}
(\tilde{A}_HP\tilde{A}_H)_{hn+r,j} &=& \sum_{l,m=0}^{kn-1}(\tilde{A}_H)_{hn+r,l}P_{l,m}(\tilde{A}_H)_{m,j}\\
&=& \sum_{l=0}^{kn-1}(\tilde{A}_H)_{hn+r,l}(\tilde{A}_H)_{p(l),j}\\
&=& \sum_{x_l=0}^{k-1}\sum_{y_l=0}^{n-1}(\tilde{A}_H)_{hn+r,x_ln+y_l}(\tilde{A}_H)_{p(x_ln+y_l),j}\\
&=& \sum_{y_l=0}^{n-1}a_{r,y_l}(\tilde{A}_H)_{p(hn+y_l),j},
\end{eqnarray*}
and similarly
\begin{eqnarray*}
(\tilde{A}_HP\tilde{A}_H)_{hn+s,j} &=& \sum_{y_l=0}^{n-1}a_{s,y_l}(\tilde{A}_H)_{p(hn+y_l),j}.
\end{eqnarray*}
Since $a_{r,y_l}=a_{s,y_l}$ by hypothesis, we get the claim.
\end{proof}

In Example \ref{esempiocaramellina} (Fig. \ref{straney}), observe that the vertices $0$ and $1$, and the vertices $4$ and $5$, have the same neighborhood in the graph $H$: this implies that the pairs of vertices $0$ and $1$; $4$ and $5$; $6$ and $7$; $10$ and $11$, have the same neighborhood in the permutational power of $H$.\\
\indent On the other hand, the condition of Proposition \ref{onlysufficient} is only sufficient, as it may occur that in the final graph two vertices share the same neighborhood, but the same property does not hold in the original graph. In Fig. \ref{figuracontroesempiovicinato}, we have represented the cyclic graph $C_8$, and its permutational $1$-st power produced by the permutation $p=(0\ 3\ 2\ 1)(4\ 5\ 6\ 7)$. One can directly check that the vertices $0$ and $4$ have the same neighborhood in the final graph, as well as the vertices $3$ and $7$; however, the equitable partition given by the neighborhood in $C_8$ is trivial, because there are no vertices sharing the same neighborhood in $C_8$.

\begin{figure}[h]
\begin{center}
\psfrag{4}{$4$}\psfrag{5}{$5$}\psfrag{6}{$6$}\psfrag{7}{$7$}
\psfrag{0}{$0$}\psfrag{1}{$1$}
\psfrag{2}{$2$} \psfrag{3}{$3$}
\includegraphics[width=0.5\textwidth]{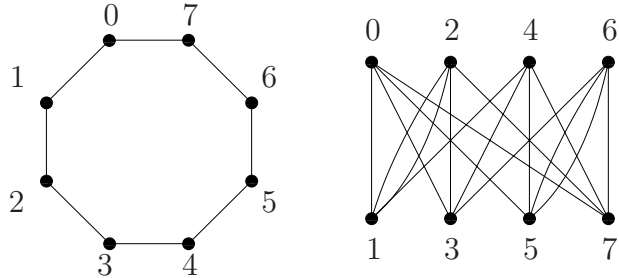} \caption{The cyclic graph $C_8$ and a permutational $1$-st power.} \label{figuracontroesempiovicinato}
\end{center}
\end{figure}

Given a matrix permutation $P$ and the partition $\hat\pi$, we define
$P/\hat\pi= M_{\hat\pi}^T P M_{\hat\pi}$ in order to describe the action of the permutation on the classes induced by $\hat\pi$. We have the following reduction result.

\begin{theorem}\label{teo1}
Let $A_H$ be the adjacency matrix of a graph $H$. Then
$$
\tilde{A}_HP_1\tilde{A}_H=\tilde{A}_HP_2\tilde{A}_H  \iff (\tilde{A}_H/{\hat\pi}) (P_1/\hat\pi ) (\tilde{A}_H/{\hat\pi})= (\tilde{A}_H/{\hat\pi}) (P_2/\hat\pi ) (\tilde{A}_H/{\hat\pi}).
$$
\end{theorem}
\begin{proof}
By Eq. \eqref{vicini},
$$
\tilde{A}_HP_i\tilde{A}_H=\tilde{A}_H M_{\hat\pi}M_{\hat\pi}^T  P_i M_{\hat\pi}M_{\hat\pi}^T \tilde{A}_H
$$
by definition of $P_i/\hat\pi$:
$$
= \tilde{A}_H M_{\hat\pi} (P_i/\hat\pi )M_{\hat\pi}^T\tilde{A}_H
$$
by equitability properties:
$$
=M_{\hat\pi} (\tilde{A}_H/{\hat\pi}) (P_i/\hat\pi ) (\tilde{A}_H/{\hat\pi}) M_{\hat\pi}^T,
$$
that is:
\begin{equation}\label{edd}
\tilde{A}_HP_i\tilde{A}_H=M_{\hat\pi} (\tilde{A}_H/{\hat\pi}) (P_i/\hat\pi ) (\tilde{A}_H/{\hat\pi}) M_{\hat\pi}^T,
\end{equation}
and then
$$
(\tilde{A}_H/{\hat\pi}) (P_1/\hat\pi ) (\tilde{A}_H/{\hat\pi})=(\tilde{A}_H/{\hat\pi}) (P_2/\hat\pi ) (\tilde{A}_H/{\hat\pi}) \implies\tilde{A}_HP_1\tilde{A}_H=\tilde{A}_HP_2\tilde{A}_H.
$$

On the other hand, by Eq. \eqref{ident},
$$
(\tilde{A}_H/{\hat\pi}) (P_i/\hat\pi ) (\tilde{A}_H/{\hat\pi})=M_{\hat\pi}^T M_{\hat\pi} (\tilde{A}_H/{\hat\pi}) (P_i/\hat\pi ) (\tilde{A}_H/{\hat\pi}) M_{\hat\pi}^T M_{\hat\pi}
$$
and  by Eq. \eqref{edd}
$$
= M_{\hat\pi}^T \tilde{A}_HP_i\tilde{A}_H M_{\hat\pi}.
$$
Therefore
\begin{eqnarray*}
(\tilde{A}_H/{\hat\pi}) (P_i/\hat\pi ) (\tilde{A}_H/{\hat\pi})= M_{\hat\pi}^T \tilde{A}_HP_i\tilde{A}_H M_{\hat\pi}
\end{eqnarray*}
and then
$$
\tilde{A}_HP_1\tilde{A}_H =\tilde{A}_HP_2\tilde{A}_H  \implies (\tilde{A}_H/{\hat\pi}) (P_1/\hat\pi ) (\tilde{A}_H/{\hat\pi})=(\tilde{A}_H/{\hat\pi}) (P_2/\hat\pi ) (\tilde{A}_H/{\hat\pi}).
$$
\end{proof}

\begin{corollary}\label{teo}
Let $A_H$ be the adjacency matrix of a graph $H$. Then
$$
\tilde{A}_HP\tilde{A}_H \mbox{ is symmetric } \iff (\tilde{A}_H/{\hat\pi}) (P/\hat\pi ) (\tilde{A}_H/{\hat\pi})\mbox{ is symmetric}.
$$
\end{corollary}
\begin{proof}
It is a particular case of Theorem \ref{teo1} when $P_1=P$ and $P_2=P^T$.
\end{proof}


\begin{corollary}\label{coro2}
If $P_1/\hat\pi=P_2/\hat\pi$  then $\tilde{A}_HP_1\tilde{A}_H=\tilde{A}_HP_2\tilde{A}_H$. In particular, if $P/\hat\pi$ is symmetric then $\tilde{A}_HP\tilde{A}_H$ is symmetric.
\end{corollary}
\begin{proof}
If $P_1/\hat\pi=P_2/\hat\pi$ then $(\tilde{A}_H/{\hat\pi}) (P_1/\hat\pi ) (\tilde{A}_H/{\hat\pi})=(\tilde{A}_H/{\hat\pi}) (P_2/\hat\pi ) (\tilde{A}_H/{\hat\pi})$, and by virtue of Theorem \ref{teo1} the first claim is proved. The second claim easily follows from Corollary \ref{teo}.
\end{proof}
The following proposition answers our Question (2) in Section \ref{sectionpreliminaries}, in the context of the neighborhood partition $\hat\pi$.

\begin{prop}\label{mm}
If $P/\hat\pi$ is symmetric then there exists $Q$ of order $2$ such that $\tilde{A}_HP\tilde{A}_H=\tilde{A}_HQ\tilde{A}_H$.
\end{prop}
\begin{proof}
Let $\hat\pi=\{C_1,\ldots, C_m\}$ be the neighborhood partition on the vertex set of $k$ disjoint copies of $H$, naturally induced by the neighborhood partition of $V_H$.
Notice that the entry $(P/\hat\pi)_{i,j}$ counts, up to the factor $1/\sqrt{|C_i||C_j|}$, the number of elements in the class $C_i$ moved by the permutation $p$ to the class $C_j$.
We define
$$
V_{i,j}^{(p)}=\{v\in C_i: \ p(v)\in C_j\},
$$
for any $i,j\in \{1,\ldots,m\}$. Therefore $(P/\hat\pi)_{i,j}=(P/\hat\pi)_{j,i}$ if and only if $|V_{i,j}^{(p)}|=|V_{j,i}^{(p)}|$ for any $i,j$. Moreover for every $i$ there exists a partition $C_i=\cup_{j=1}^m V_{i,j}^{(p)}$. Observe that some $V_{i,j}^{(p)}$ can be empty. Since $P/\hat\pi$ is supposed to be symmetric, we have $|V_{i,j}^{(p)}|=|V_{j,i}^{(p)}|$, so that we can define a bijection $\sigma$ between these two sets. Since we can do this for any $i$ and $j$ we can extend the bijection to any of the parts $C_i$. If $\sigma(x)=y$ then we put $q(x)=y$ and $q(y)=x$. On the sets $V_{i,i}^{(p)}$ we can choose $\sigma$ to be the identity map. The $q$ we get is a permutation of order $2$ on the set of vertices. It is clear that, in this way $|V_{i,j}^{(q)}|=|V_{i,j}^{(p)}|=|V_{j,i}^{(p)}|=|V_{j,i}^{(q)}|$ and this means that $P/\hat\pi=Q/\hat\pi$. The statement follows by applying Theorem \ref{teo1}.
\end{proof}

\begin{remark}
The converse of Corollary \ref{coro2} is false. In fact, we have seen in Example \ref{ciclo88} that, in the case of the cycle $C_8$, the partition $\hat\pi$ is trivial, so that $P/\hat\pi=P$ for any matrix permutation $P$, but we showed that there exist nonsymmetric permutation matrices $P$ such that $A_{C_8}PA_{C_8}$ is symmetric. Moreover, the converse of Proposition \ref{mm} is false, since we showed in Example \ref{ciclo88} that there exist nonsymmetric permutation matrices $P$ and symmetric permutation matrices $Q$ such that  $A_{C_8}PA_{C_8}=A_{C_8}QA_{C_8}$.
\end{remark}

By arguing as in the proof of Proposition \ref{mm}, we are able to give an estimate of the number of permutations of order $2$ giving rise to the same graph in the case in which the matrix $P/\hat\pi$ is symmetric. Put $p_{i,j}=|V_{i,j}^{(p)}|$.

\begin{prop}
Let $p_{i,j}/\sqrt{|C_i||C_j|}$ be the entry $(P/\hat\pi)_{i,j}$ of the matrix $P/\hat\pi$, with $i,j=1,\ldots, m$. There are $\prod_{i,j=1}^m p_{i,j}!$ permutation matrices $Q$ of order $2$ such that $\tilde{A}_HP\tilde{A}_H=\tilde{A}_HQ\tilde{A}_H$.
\end{prop}
\begin{proof}
Since the matrix $P/\hat\pi$ is symmetric, we have $p_{i,j}=|V_{i,j}^{(p)}|=|V_{j,i}^{(p)}|=p_{j,i}$. From the proof of Proposition \ref{mm}, it is enough to count the number of possible bijections from the set $V_{i,j}^{(p)}$ to $V_{j,i}^{(p)}$ for any $i$ and $j=i,\ldots, m$, i.e., the number of bijections of $V_{i,j}^{(p)}$ to itself. Fixed $i$ and $j$ this number is $p_{i,j}!$. The result follows.
\end{proof}

\begin{example}
Consider the graph $H$ and its permutational $2$-nd power in Example \ref{esempiocaramellina} induced by the permutation $p=(0\ 4\ 6\ 3\ 7\ 5\ 1\ 8)(2\ 9)(10\ 11)$ with associated permutation matrix $P$. If $M_{\hat{\pi}}$ is the characteristic matrix of the neighborhood partition on two copies of $H$, one gets:
$$
P/\hat{\pi} = M_{\hat{\pi}}^T P M_{\hat{\pi}} = \left(
                                              \begin{array}{cccccccc}
                                                0 & 0 & 0 & 1/2 & 0 & 1/\sqrt{2} & 0 & 0 \\
                                                0 & 0 & 0 & 0 & 0 & 0 & 1 & 0 \\
                                                0 & 0 & 0 & 0 & 1/\sqrt{2} & 0 & 0 & 0 \\
                                                1/2 & 0 & 0 & 0 & 1/2 & 0 & 0 & 0 \\
                                                0 & 0 & 1/\sqrt{2} & 1/2 & 0 & 0 & 0 & 0 \\
                                                1/\sqrt{2} & 0 & 0 & 0 & 0 & 0 & 0 & 0 \\
                                                0 & 1 & 0 & 0 & 0 & 0 & 0 & 0 \\
                                                0 & 0 & 0 & 0 & 0 & 0 & 0 & 1 \\
                                              \end{array}
                                            \right)
$$
which is a symmetric matrix. By virtue of Proposition \ref{mm}, there exists a matrix permutation $Q$ of order $2$ such that $(I_2\otimes A_H)P(I_2\otimes A_H)=(I_2\otimes A_H)Q(I_2\otimes A_H)$. One can directly check that a matrix $Q$ of order $2$ satisfying this property is the permutation matrix associated with the permutation $q=(0\ 4)(1\ 8)(2\ 9)(3\ 7)(5\ 6)(10\ 11)$.
\end{example}

\begin{theorem}\label{mmm}
If $\tilde{A}_H/\hat\pi$ is invertible and $\tilde{A}_HP\tilde{A}_H$ is symmetric, then there exists $Q$ of order $2$ such that $\tilde{A}_HP\tilde{A}_H=\tilde{A}_HQ\tilde{A}_H$.
\end{theorem}
\begin{proof}
If $\tilde{A}_HP\tilde{A}_H$ is symmetric, then by Corollary \ref{teo} the matrix $(\tilde{A}_H/{\hat\pi})(P/\hat\pi)(\tilde{A}_H/{\hat\pi})$ is symmetric; now, since $\tilde{A}_H/{\hat\pi}$ is invertible, we deduce that $P/\hat\pi $ is symmetric. We can now apply Proposition \ref{mm} to get the claim.
\end{proof}

Our results are useful when we analyze graphs with an adjacency matrix whose singularity depends on the repetition of some rows. The extreme example is treated in the following section.

\subsection{The complete bipartite graph}\label{bipart}
Let $K_{m,n}$ denote the complete bipartite graph on $m+n$ vertices, whose vertex set is partitioned into two sets $V_1$ and $V_2$, with $|V_1|=m$ and $|V_2|=n$, such that every vertex of $V_i$ is connected by an edge to every vertex of $V_j$, with $i\neq j$, and no edge connects vertices belonging to the same part. Let $J_{m,n}$ denote the $m\times n$ matrix whose entries are all equal to $1$. Then, up to a reordering of the vertices, the adjacency matrix of $K_{m,n}$ is
$$
A_{K_{m,n}} = \left(
                \begin{array}{c|c}
                  O & J_{m,n} \\
                  \hline
                  J_{n,m} & O \\
                \end{array}
              \right).
$$
Let us number the vertices of $V_1$ by $1,2,\ldots, m$ and the vertices of $V_2$ by $m+1,\ldots, m+n$. It is straightforward that the  neighborhood partition $\hat{\pi}$ of the vertices of $K_{m,n}$ coincides with the partition $V_{K_{m,n}} = V_1 \sqcup V_2$. In particular, the characteristic matrix $M_{\hat{\pi}}$ is the $(m+n)\times 2$ matrix
$$
M_{\hat{\pi}} = \left(
                      \begin{array}{cc}
                        1/\sqrt{m} & 0 \\
                        \vdots & \vdots \\
                         1/\sqrt{m} & 0 \\
                        0 &  1/\sqrt{n} \\
                        \vdots & \vdots \\
                        0 &  1/\sqrt{n} \\
                      \end{array}
                    \right)
$$
and it satisfies the following equalities:
$$
M_{\hat{\pi}}^TM_{\hat{\pi}} =I_2; \qquad  M_{\hat{\pi}}M_{\hat{\pi}}^T = \left(
                                                                                                  \begin{array}{c|c}
                                                                                                    \frac{1}{m}J_{m,m} & 0 \\   \hline
                                                                                                    0 &  \frac{1}{n}J_{n,n} \\
                                                                                                  \end{array}
                                                                                                \right);  \qquad    M_{\hat{\pi}}^TA_{K_{m,n}}M_{\hat{\pi}} = \left(
                                                                                                  \begin{array}{cc}
                                                                                                    0 & \sqrt{mn} \\
                                                                                                    \sqrt{mn} & 0 \\
                                                                                                  \end{array}
                                                                                                \right).
$$
The latter matrix can be regarded, up to normalization, as the adjacency matrix of the graph $K_{m,n}/\hat{\pi}$, which reduces to the complete graph on $2$ vertices.

Now let $k\geq 1$ be a positive integer and consider the disjoint union of $k$ copies of $K_{m,n}$, so that the adjacency matrix of this new graph is given by $\tilde{A}_{K_{m,n}}=I_k\otimes A_{K_{m,n}}$, and the associated characteristic matrix is $\tilde{M}_{\hat{\pi}} = I_k\otimes M_{\hat{\pi}}$. In particular:
$$
\tilde{A}_{K_{m,n}}/\hat{\pi} = \tilde{M}_{\hat{\pi}}^T\tilde{A}_{K_{m,n}}\tilde{M}_{\hat{\pi}}  = I_k \otimes \left(
                                                                                                  \begin{array}{cc}
                                                                                                    0 & \sqrt{mn} \\
                                                                                                    \sqrt{mn} & 0 \\
                                                                                                  \end{array}
                                                                                                \right).
$$
Now let $p$ be a permutation on $k(m+n)$ elements, and let $P$ be the corresponding permutation matrix. As the matrix $\tilde{A}_{K_{m,n}}/\hat{\pi}$ is nonsingular, by virtue of Corollary \ref{teo}, the matrix $\tilde{A}_{K_{m,n}}P\tilde{A}_{K_{m,n}}$ is symmetric if and only if the matrix $P/\hat{\pi} = \tilde{M}_{\hat{\pi}}^T P \tilde{M}_{\hat{\pi}}$ is symmetric. Therefore, by applying Proposition \ref{mm}, we get the following theorem.
\begin{theorem}\label{bipteo}
A permutational $k$-th power of the graph $K_{m,n}$ can always be obtained by a classical zig-zag product. That is,
if a permutation matrix $P$ is such that $\tilde{A}_{K_{m,n}}P\tilde{A}_{K_{m,n}}$ is symmetric, then there exists a symmetric permutation matrix $Q$ such that
$\tilde{A}_{K_{m,n}}P\tilde{A}_{K_{m,n}}=\tilde{A}_{K_{m,n}}Q\tilde{A}_{K_{m,n}}$.
\end{theorem}

In order to investigate the symmetry of $P/\hat{\pi}$, observe that $P/\hat{\pi}$ is a square matrix of size $2k$, whose rows and columns can be indexed by the pairs $(x,y)$, where the copy index $x$ varies in $\{1,\ldots, k\}$ and the part index $y$ varies in $\{1,2\}$. Therefore, one has symmetry if and only if the number of elements in the part $y$ of the copy $x$ which are sent to elements in the part $y'$ of the copy $x'$ equals the number of elements in the part $y'$ of the copy $x'$ which are sent to elements in the part $y$ of the copy $x$. Here below an explicit example in the case $k=2, m=3, n=5$.

\begin{example}
Consider the following permutation of $16$ elements
$$
p = \left(
      \begin{array}{ccc|ccccc|ccc|ccccc}
        1 & 2 & 3 & 4 & 5 & 6 & 7 & 8 & 9 & 10 & 11 & 12 & 13 & 14 & 15 & 16 \\
        4 & 10 & 16 & 8 & 2 & 9 & 15 & 14 & 6 & 13 & 3 & 7 & 11 & 5 & 12 & 1 \\
      \end{array}
    \right),
$$
with associated permutation matrix $P$. It is easy to check that:
$$
\tilde{M}_{\hat{\pi}} = I_2 \otimes \left(
                                                                                               \begin{array}{cc}
                                                                                                 1/\sqrt{3} & 0 \\
                                                                                                 1/\sqrt{3} & 0 \\
                                                                                                 1/\sqrt{3} & 0 \\
                                                                                                 0 & 1/\sqrt{5} \\
                                                                                                 0 & 1/\sqrt{5} \\
                                                                                                 0 & 1/\sqrt{5} \\
                                                                                                 0 & 1/\sqrt{5} \\
                                                                                                 0 & 1/\sqrt{5} \\
                                                                                               \end{array}
                                                                                             \right)
\qquad P/\hat{\pi} = \left(
                                              \begin{array}{cccc}
                                                0 & 1/\sqrt{15} & 1/3 & 1/\sqrt{15} \\
                                                1/\sqrt{15} & 1/5 & 1/\sqrt{15} & 2/5 \\
                                                1/3 & 1/\sqrt{15} & 0 & 1/\sqrt{15} \\
                                                1/\sqrt{15} & 2/5 & 1/\sqrt{15} & 1/5 \\
                                              \end{array}
                                            \right)
$$
and therefore the matrix $(I_2\otimes A_{K_{3,5}})P(I_2\otimes A_{K_{3,5}})$ is symmetric. A permutation $Q$ of order $2$ satisfying the property
$(I_2\otimes A_{K_{3,5}})P(I_2\otimes A_{K_{3,5}})= (I_2\otimes A_{K_{3,5}})Q(I_2\otimes A_{K_{3,5}})$ is $Q= (1\ 5)(2\ 11)(3\ 16)(6\ 9)(7\ 12)(8\ 14)(10\ 13)$, constructed as in the proof of Proposition \ref{mm}.
\end{example}


\end{document}